\documentclass[leqno]{aadmbook}
\usepackage{amsthm}
\usepackage{amsfonts}
\usepackage{amsmath}
\begin{document}
\newcommand{\R}{{\mathbb R}}
\newcommand{\N}{{\mathbb N}}
\newcommand{\C}{{\mathbb C}}
\newcommand{\Z}{{\mathbb Z}}
\newcommand{\h}{{\delta}}
\newtheorem{lemma}{Lemma}
\newtheorem{theorem}{Theorem}
\newtheorem*{corollary}{Corollary}
\newtheorem*{fact}{Theorem} 
\newtheorem*{lfact}{Lemma}
\newtheorem{conjecture}{Conjecture}
\newtheorem{proposition}{Proposition}
\theoremstyle{remark}
\newtheorem{example}{Example}
\newtheorem{problem}{Problem}
\newtheorem*{remark}{Remark}
\newtheorem*{definition}{Definition}
\oddsidemargin 16.5mm
\evensidemargin 16.5mm
\thispagestyle{plain}

\begin{center}
{\large \sc  Applicable Analysis and Discrete Mathematics}

{\small available online at  http:/$\!$/pefmath.etf.rs }
\end{center}

\noindent{\small{\sc  Appl. Anal. Discrete Math.\ }{\bf 7} (2013),
143--160.} \hfill{\scriptsize doi:10.2298/AADM121204025G}

\vspace{5cc}
\begin{center}

{\large\bf STABILITY ESTIMATES FOR DISCRETE HARMONIC FUNCTIONS ON PRODUCT DOMAINS
\rule{0mm}{6mm}\renewcommand{\thefootnote}{}
\footnotetext{

\rule{2.4mm}{0mm}Keywords and Phrases. discrete harmonic functions, Phragm\'{e}n-Lindel\"{o}f principle, eigenvalues of Dirichlet-Laplacian, three-line theorem, stability estimates.
}}

\vspace{1cc}
{\large\it Maru Guadie}

\vspace{1cc}
\parbox{24cc}{{\small

We study the Dirichlet problem for discrete harmonic functions in unbounded product domains on multidimensional lattices. First we prove some versions of the  Phragm\'{e}n-Lindel\"{o}f theorem and use Fourier series to obtain a discrete analog of the three-line theorem for the gradients of harmonic functions in a strip. Then we derive estimates for the discrete harmonic measure and  use elementary spectral inequalities to obtain stability estimates for Dirichlet problem in cylinder domains.}}
\end{center}
\vspace{1cc}

\vspace{1.5cc}
\begin{center}
{\bf 1. INTRODUCTION}
\end{center}

We consider functions defined on subsets of  the multidimensional lattice $(\delta\Z)^{m}$  in $\R^m$.  The usual 
 $2m+1$-point discretization of the Laplace operator is denoted by $\Delta_m$ or $\Delta_{\delta,m}$ to emphasize the mesh of the lattice, the accurate definition is given below. Then we study the following Dirichlet problem 
\begin{eqnarray*}
&\Delta_m u=0,\\
&u=f\ {\text{on}}\ \partial D,\\
&u\in H_b(D),
\end{eqnarray*}
where $H_b(D)$ is some class of functions of bounded growth in $D$,  and  $D$ is an unbounded connected (on the lattice) subset of $(\delta\Z)^m$. Our main question is for which $H_b(D)$ the problem above has a unique solution. Moreover, when the solution is unique we estimate how the error in the boundary data affects the error of the solution. Such estimates are called conditional stability estimates, we suppose a priori that solution belongs to $H_b(D)$.  Since our problem is linear, stability estimate reduces to a bound of some norm of the solution $u\in H_b(D)$ by some norm of its boundary values $f$. 

First, we prove that if $D=\Omega\times \R^k,$ where $\Omega$ is a bounded domain in $\R^n,$ $u(x,y)$ is a discrete harmonic function in $D\cap \left(\delta\Z\right)^{n+k}$ that satisfies 
\[|u(x,y)|\le C\exp\left(c\|y\|_1\right)\]
 for some $c=c(\Omega,k),$   and $u=0$ on $\partial D$ then  $u=0$ (here and in what follows $\|y\|_1=|y_1|+...+|y_k|,$ and  $\left\|y\right\|_{\infty}=\max\left\{\left|y_{1}\right|,\ldots,\ \left|y_{k}\right|\right\}$ where $y=(y_1,...,y_k)\in \R^k$). We refer to this statement as a discrete version of the Phragm\'{e}n-Lindel\"{o}f theorem, it implies the uniqueness in the Dirichlet problem in the class of functions of limited growth.
 
  We consider more carefully the case $\Omega=[0,1]$ and solve the Dirichlet problem using Fourier analysis when the  boundary data is in $l^2$. We obtain
\[\|u(x,.)\|_{l^2}\le \|f\|_{l^2}.\] 
We  also use this technique to show that gradients of discrete harmonic functions satisfy the following  three-line inequality that resembles three-line theorem of Hadamard,
\begin{equation}\label{eq:last1}
\|\nabla u(\delta k,\cdot)\|_{l^2(\Z^k)}\le (\|\nabla u(0,\cdot)\|)^{1-\frac{k}{M}}(\|\nabla u(\delta M,\cdot)\|)^{\frac{k}{M}},
\end{equation}
where $(M+1)\delta=1$. Both the Phragm\'{e}n-Lindel\"{o}f theorem and Hadamard's three line theorem are classical results in complex analysis (for example see \cite{SS}). We discuss discrete version of their multidimensional generalizations, corresponding continuous results are known and we provide the references throughout the text.

Finally, to obtain conditional stability estimates for Dirichlet problem with partial boundary data (see Theorem \ref{th:last}), we study the discrete harmonic measure in the truncated cylinder  $\Omega\times[-N,N]$. We also use elementary properties of the spectrum of the discrete Dirichlet problem for the Laplacian on $ \Omega$ and some comparison results that can be found in T. Biyikoglu, J. Leydold, P. F. Stadler  \cite{BLS07} and D.~Cvetkovi\'{c}, P. Rowlinson, S. Simi\'{c}   \cite{CRS10}.

The article is organized as follows. In the next section we give necessary definitions and results for discrete harmonic functions, including basic properties of the eigenvalues and eigenfunctions of the discrete Laplace operator with Dirichlet boundary condition. We also prove a simple version of the Phragm\'{e}n-Lindel\"{o}f theorem for product domains. In Section 3 we use Fourier analysis to study discrete harmonic functions in a strip, in particular we obtain the logarithmic convexity inequality (\ref{eq:last1}). Our main stability result for the Dirichlet problem in an infinite cylinder is proved in the last section, it follows from estimates of discrete harmonic measure and a more accurate version of the Phragm\'{e}n-Lindel\"{o}f theorem. 

\vspace{1.5cc}
\begin{center}
{\bf 2. PRELIMINARIES}
\end{center}

\noindent{\bf 2.1 Discrete harmonic functions}

The theory of discrete harmonic functions on the lattices dates back to at least as early as 1920s, when fundamental works of H. B. Phillips and N. Wiener  \cite{PW23} ,  and R. Courant, K. Friedrichs, and H. Lewy  \cite{CFL28}  were published. In the middle of the last century an important contribution to the theory of discrete harmonic functions was done by H. A. Heilbronn \cite{H49}  and R. J. Duffin \cite{D53}. One of the original motivations for the study of discrete harmonic functions is that such functions converge to continuous ones. For example to obtain a solution of the Dirichlet problem one may solve discrete problems in lattice domains and pass to the limit as the mesh size of the lattice goes to zero, we refer the reader to the classical works mentioned above and to the article of I.~G.~Petrowsky \cite{P41}.   

Suppose that $u(x)$\ is a function defined on a subset of the lattice  $(\h\Z)^{m}$. Then the $\h$-discrete Laplacian of $u$ is defined by
\begin{equation*}
\Delta_{\h} u(x)=\Delta_{\h,m} u(x)=\h^{-2}\left(\sum^{m}_{j=1}\left(u(x+\h e_{j})+u(x-\h e_{j})\right)-2mu(x)\right),
\end{equation*}
where $e_1, e_2,..., e_m$ is the standard coordinate basis for $\Z^m$ and    $-\Delta_{\delta}$ coincides with the combinatorial Laplacian of the lattice where the conductance associated to each edge equals $\delta^{-2}.$  This is the discrete version of the Laplace-Beltrami operator in Riemannian manifolds. We refer the reader to T. Biyikoglu, J. Leydold, P. F. Stadler  \cite{BLS07} for the details. Potential theory on finite networks is an active area of investigation, see for example\cite {BCE2005}  and references therein.
\begin{definition}
A function $u$ is called $\h$-discrete harmonic at a point $x$ of the lattice   $(\h\Z)^{m}$ if it is define at $x$ together with all its neighbors and satisfies the  equation 
\[
\Delta_{\h} u(x)=0.
\]
So the value of a discrete harmonic function at a lattice point is the average of its values at the $2m$ neighboring points.
\end{definition}

Discrete harmonic functions share many properties of continuous ones. For example results on the maximum principle, solution to the Dirichlet problem, Green's function, and Liouville's theorem can be found is the very first articles on the subject, see also Y. Colin de Verdi\'{e}re  \cite{CV93} and C. Kiselman \cite{Kis05} for more recent surveys and more general discrete structures. On the other hand not all results about continuous harmonic functions are easily generalized to the discrete case. For example zero sets of discrete harmonic functions are difficult to compare to those of continuous ones. For any finite square there exists a discrete harmonic polynomial that vanishes at each lattice point of this square. We study growth properties of discrete harmonic functions in cylinders and strips and provide accurate estimates that show to which extend continuous theorems can be generalized to solutions of the discrete equation that arises in the simplest numerical scheme. 

We consider discrete harmonic functions on subsets of $(\h\Z)^{m},$  $D^{\delta}\subset \left(\delta\Z\right)^{m}$ is called a (discrete) domain if it is connected, i.e., for any two points $x$ and $y$ in $D^{\delta}$ there exists a sequence $\left\{x_{0},x_{1},\ldots,x_{s}\right\}$ such that $x_{0}=x,\ x_{s}=y,\ x_{j}\in D^{\delta},\ x_{j}$ and $x_{j+1}$ are neighboring points of the lattice $(\h\Z)^m$.

A point $x\in (\h\Z)^m\setminus D^{\delta}$ is called a boundary point of $D^{\delta}$ if at least one of the $2m$ neighbors of $x$ is in $D^{\delta}.$ We denote the set of boundary points of $D^{\delta}$ by $\partial D^{\delta}$, we also use the notation $\overline{D}^{\delta}=D^{\delta}\cup\partial D^{\delta}$.  A domain is called finite if it contains only finite number of points, otherwise it is called infinite.
  \begin{definition}
A function $u$ defined on $D^{\delta}\cup\partial D^{\delta}$ is called $\h$-discrete subharmonic (superharmonic) in $D^{\delta}$ if $\Delta_{\h}u\geq 0 \ (\leq 0)$ in $D^{\delta}.$
\end{definition}
Clearly, a function is $\delta$-discrete harmonic in $D^{\delta}$ if it is both $\delta$-discrete subharmonic and superharmonic.
The following maximum principle holds (see for example  \cite{Kis05}).
\begin{fact}
If $u$ is $\delta$-discrete subharmonic in a finite domain $D$ then 
\[\max_{\overline{D}}u=\max_{\partial D}u.\]
\end{fact}

 Simple examples show that the maximum principle does not hold for infinite domains. 

\noindent{\bf 2.2 Eigenvalues and eigenfunctions for the discrete Laplacian}

In order to prove a version of the Phragm\'{e}n-Lindel\"{o}f theorem for discrete subharmonic functions in cylindrical domains, we need some basic facts about  eigenfunctions and eigenvalues of the discrete Dirichlet problem for the Laplacian on the base of the cylinder. For more general theory of graph spectra we refer the reader to D.~Cvetkovi\'{c}, P.~Rowlinson, S.~Simi\'{c} \cite{CRS10} and F. R. K. Chung \cite[ch 1]{RKC97}.

Throughout the paper $\Omega$ denotes  a bounded domain in $\R^{n}$ $n\ge 1,$ with Lipschitz boundary and  $\Omega^\h=\Omega\cap(\h\Z)^{n}$. We always assume that $\h<\h_0$ is small enough such that $\Omega^\h$ is a discrete connected set. We study $\h$-discrete harmonic functions that are defined on the product domain $D^\h(\Omega)=\overline{\Omega}^{\delta}\times (\h\Z)^k$ and vanish on the boundary. We consider the eigenvalues $\left\{\lambda_j(\Omega)\right\}$ of the continuous $n$-dimensional Dirichlet problem for the Laplacian on $\Omega$  and  the eigenvalues of the corresponding discrete operators. We denote the eigenvalue for the discrete Dirichlet problem on $\overline{\Omega}^{\delta}$  by  $\lambda^{\delta}_{j}\left(\Omega^{\delta}\right)$  and we use the notation $\lambda^{\delta}_{j}$  when it does not lead to confusion.  It is known (see for example  \cite{BLS07} or   \cite{CRS10}) that the eigenvalues of the following problem
\[
\left\{ \begin{array}{cc}
-\Delta_{\h,n} f=\lambda f\quad &{\text{in}}\ \Omega^\h\\
f=0\quad&{\text{on}}\  \partial \Omega^\h\end{array}\right. 
\]
are positive, $0<\lambda^\h_{1}<\lambda^\h_{2}\leq\ldots\leq\lambda^\h_{K^\h},$ the first eigenvalue is simple and the corresponding eigenfunction $f^\h_1$ can be chosen strictly positive in $\Omega^\h$. The last statement is an analog of the classical result on the first eigenfunction of Dirichlet problem for the Laplacian, see R. Courant, D. Hilbert \cite[\S 6, ch VI]{CH53}. For the discrete operator it follows from the Perron-Frobeniuos theorem on positive matrices, see for example  \cite[Corollary 2.23]{BLS07} . Clearly $K^\h$ is finite in the discrete case and equals the number of points of $\Omega^\h$. 

It is also known that $\lambda_k^\h(\Omega^\h)\rightarrow \lambda_k(\Omega)$ as $\h\rightarrow 0$. We don't discuss the limits arguments in this article, but we indicate which of our estimates survive the limit passage as $\delta\rightarrow 0$.  

The eigenvalues  $\lambda^{\h}_{k}(\Omega^{\h})$ are given by the following minimax principle, see  \cite[Corollary 2.6]{BLS07}, 
\[\lambda^{\h}_{k}(\Omega^{\h})=\min_{w\in W_{k}}\max_{0\neq g\in w}\frac{\left\langle g,L^\h_{\Omega}g\right\rangle}{\left\langle g,g\right\rangle},\]
where $W_{k}$  denotes the set of subspaces of dimension at  least $k$ and   $L^\h_{\Omega^{\delta}}$ is the $\h$-discrete Laplacian of $\Omega$ with Dirichlet boundary condition.   This readily implies that if $\Omega^{\prime}\supset \Omega$ then  
\begin{equation}\label{eq:incl}
\lambda^{\h}_{k}(\Omega^{\prime})\leq \lambda^{\h}_{k}(\Omega).
\end{equation}
We denote by $N_\Omega^\delta$ the counting function, $N_\Omega^\delta(\lambda)$ equals the number of eigenvalues $\lambda_k^\delta(\Omega)$ that are less than or equal to $\lambda$. Then (\ref{eq:incl}) implies 
\begin{equation}\label{eq:incl1}
N_{\Omega'}^\delta(\lambda)\ge N_{\Omega}^\delta(\lambda).
\end{equation}

\noindent{\bf 2.3 Eigenvalues for the cube}

We need some estimates of the growth of the eigenvalues $\lambda_j^\h(\Omega)$ to prove a precise version of the  Phragm\'{e}n-Lindel\"{o}f theorem in the last section of this article. We obtain them by comparing the eigenvalues to those of a large cube $Q$ containing $\Omega$. The latter can be found explicitly. Let $Q_R=(0,R)^n$, where $R\in\N$ and let $M=1/\h\in\N$. We consider the following problem
\[
\left\{\begin{array}{cc}
-\Delta_{\h,n} f=\lambda f\quad &{\text{in}}\ Q_R^\h\\
f=0 \ \quad &{\text{on}}\ \partial Q_R^\h\end{array}\right.
 \]
This is an eigenvalue problem for a matrix of the size $(R\h^{-1}-1)^{n}\times(R\h^{-1}-1)^{n}$.
Let $J=\{1,2,..., R\h^{-1}-1\}$, for any $\overline{k}\in J^{n}$, $\overline{k}=(k_1,...,k_{n})$  the function 
 \[f_{\overline{k}}(x_1,...,x_{n})=\prod_{j=1}^{n}\sin\frac{k_j\pi}{R}x_j\] 
is an eigenfunction and the corresponding eigenvalue is 
 \[\lambda^\h_{\overline{k}}=2\h^{-2}\left(n-\sum_{j=1}^{n}\cos\frac{k_j\pi \h}{R}\right).\] 
 Using the elementary inequality $1-\cos x\ge 2\pi^{-2}x^2$, when $x\in(0,\pi)$ we obtain
 \[\lambda^\h_{\overline{k}}\ge 4R^{-2}\sum_{j=1}^{n}k_j^2.\] For the details of eigenvalues and eigenfunctions for the cube we refer the reader to F. R. K. Chung \cite[ch 1]{RKC97}.
 
  The following  inequality for the counting function for the cube follows
 \begin{equation}\label{eq:Wa}
 N_{Q_R}^\h(\lambda)\le C_n(R)(\lambda^{n/2}+1),
\end{equation}
where the constant does not depend on $\h$. This inequality is an illustration of the Weyl's asymptotic for the counting function for eigenvalues of Dirichlet problem for the Laplacian.

\noindent{\bf 2.4 Phragm\'{e}n-Lindel\"{o}f theorems in cylindrical domains}

Let $\Omega$ be a bounded subdomain of $\R^{n}$  and   $D^{\delta}=\Omega^{\delta}\times(\delta\Z)^k$. Clearly, \[\Delta_{\delta,n+k}u(x,y)=\Delta_{\delta,n}u(x,y)+\Delta_{\delta,k}u(x,y),\] where the first Laplacian is taking with respect to $x$-variables and the second with respect to $y$-variables. Let  $f^\h_1$ be the first eigenfunction of the Dirichlet problem for the Laplacian in $\Omega^\h$ defined above. As we noted, $f^\h_1$ is strictly positive on $\Omega^\h,$ and we have the following positive harmonic function in $D^\h$
\[
u^\h(x,y)=f_1^\h(x)\cosh b_{\delta}y_1\cosh b_{\delta} y_2...\cosh b_{\delta} y_k,
\]
where $b_{\delta}$ is the only positive solution of
\begin{equation}\label{eq:st1k}
\cosh \delta b_{\delta}=1+\frac{1}{2k}\delta^{2}\lambda_1^\delta.
\end{equation} 
 In the discrete setting the function $f^\h_{1}$ is strictly positive; this makes the proof of our first theorem of Phragm\'{e}n-Lindel\"{o}f type more simple than the proof of a similar result for continuous functions, see for example I. Miyamoto \cite{Mi96},  F.~T.~Brawn \cite{B71} and D. V. Widder \cite{W61}.

\begin{theorem}\label{th:th1}
Let $v$ be a $\h$-discrete subharmonic function in $D^\h$ such that $v\leq 0$ on $\partial \Omega^\h\times(\h\Z)^k.$ Let $\lambda^\h_1$ be the first eigenvalue of the $\h$-discrete Dirichlet problem for the Laplacian in $\Omega$ and $b_{\delta}$ be the positive solution to the equation (\ref{eq:st1k}).
 Suppose that
 \[v(x,y)\le o(1)\exp(b_{\delta}\|y\|_1),\quad{\text{when}}\ \|y\|_1\rightarrow \infty.\]
Then $v\leq 0$ on $D^\h.$
\end{theorem}

\begin{proof}
We want to compare $v(x,y)$ to a multiple of $u^\h(x,y)$  on $\overline{\Omega}^{\delta}\times\left[-N,N\right]^k.$
On the part of the boundary $\partial\Omega^\h\times(\h\Z)^{k}$  we have $v\leq0$ and  $u^\h=0$ because $f^{\delta}_{1}=0$ on $\partial\Omega^\h.$ On the other part of the boundary, $\|y\|_1\ge N$ and 
 \[v(x,y)\leq C_{N}\exp(b_{\delta} \|y\|_1)\leq \frac{2^kC_{N}}{\min_{\Omega^\h}f^\h_{1}} u^\h(x,y),\] where $C_{N}\rightarrow 0$ as $N\rightarrow \infty.$
 
The maximum principle for subharmonic functions implies that
 \[v(x,y)\leq \frac{2^kC_{N}}{\min_{\Omega^\h}f^\h_{1}}  u^\h(x,y),\quad{\text{where}}\ x\in\Omega^\h,\ y\in (\delta\Z)^k,\ \left\|y\right\|_\infty\le N.\]
 Now if we fix $(x,y)$ and let $N$ grow to infinity, we obtain
 $v(x,y)\leq 0.$
\end{proof}

The theorem holds for subharmonic functions with all estimates from above only. If we have a discrete harmonic function $h$ and apply the above statement to $h$ and $-h$ we obtain the uniqueness for the Dirichlet problem in $D^\h$ in the class of functions 
\[
H_b(D^\h)=\{u:D^\h\rightarrow \R: |u(x,y)|=o(\exp(b_{\delta}\|y\|_1)), \left\|y\right\|_{1}\rightarrow\infty\}.\] 
\begin{corollary}\label{c:m}
Let $u$ and $v$ be $\h$-discrete harmonic functions on $D^\h,$  $ u,v\in H_b(D^\h)$. If $u=v$ on $\partial(\Omega^\h)\times(\h\Z)^k$ then $u=v$ on $D^\h.$
\end{corollary}
\begin{proof}
Let $g=u-v.$ Then $g$ is $\h$-discrete harmonic in  $D^\h$ and $g=0$ on $\partial(\Omega^\h)\times(\h\Z)^k$.  Moreover $\left|g(x,y)\right|\leq \left|u(x,y)\right|+\left|v(x,y)\right|$ and therefore 
\[
\left|g(x,y)\right|\leq C_N\exp(b_{\delta}\|y\|_1),\quad 
{\text{when}}\  \|y\|_1\ge N,\] where $C_{N}\rightarrow 0$ as $N \rightarrow \infty.$
Then $g\le 0$ on $D^\h$ by Theorem \ref{th:th1}. In the same way we obtain $-g\leq 0$ and thus $u=v$.
\end{proof}
We note that $b_{\delta}\rightarrow\sqrt{\lambda_1(\Omega)/k}$ when $\h\rightarrow 0$, however Theorem \ref{th:th1} does not survive a limit argument as $\h\rightarrow 0$. In the last section we provide an estimate for $\h$-discrete harmonic functions in truncated cylinders that allows us to prove a more accurate version of the Phragm\'{e}n-Lindel\"{o}f theorem. 


\vspace{1.5cc}
\begin{center}
{\bf 3. DISCRETE HARMONIC FUNCTIONS ON STRIPS}
\end{center}

In this section we study quantitative uniqueness for discrete harmonic functions and their gradients on strips $S=(0,1)\times\R^n$. We remark that eigenvalues of the discrete Dirichlet problem for the Laplacian on $[0,1]^\delta$ are 
\[\lambda_l^\delta=2\delta^{-2}(1-\cos 2\pi l\delta).\]
 In particular the   Phragm\'{e}n-Lindel\"{o}f theorem proved in the last section implies the uniqueness in the Dirichlet problem for discrete harmonic functions that satisfy
\begin{equation}\label{eq:st1a}
|u(x,y)|=o\left(\exp(b_{\delta}\|y\|_1)\right),\quad \left\|y\right\|_{1}\rightarrow\infty,
\end{equation}
 where 
\[
\cosh \delta b_{\delta}=  \frac{n+1}{n}-\frac1{n}\cos 2\pi\delta.
\]

\noindent{\bf 3.1 Tempered harmonic functions in a strip}

Now we consider \textit{tempered} harmonic functions in the strip and use the Fourier representation to solve the Dirichlet problem.  
\begin{definition}
Let $u$ be a $\delta-$discrete  function on $S^{\delta},$ $u$ is said to be tempered if
\[\sum^{1/\delta}_{k=0}\sum_{j\in \Z^n}\left|u(\delta k,\delta j)\right|^{2}<\infty.\]
\end{definition}
\begin{theorem}
Let $u$ be a  $\delta$-discrete harmonic function in $S^\delta$ such that (\ref{eq:st1a}) holds.   Suppose that 
\[\sum_{j\in \mathbb{Z}^n}\left|u\left(0,\delta j\right)\right|^{2}<\infty\ \mbox{ and}\  \sum_{j\in \mathbb{Z}^n}\left|u\left(1,\delta j\right)\right|^{2}<\infty.\] 
Then  $\left\{u\left(\delta k,\delta j\right)\right\}_{j\in \mathbb{Z}^n}\in l^{2}\left(\mathbb{Z}^n\right)$ for each $k=1,\ 2,\ \ldots,\ L-1,$ $i.e,$ $u$ is tempered, moreover
\[
\sum_{j\in \Z^n}|u(\delta k,\delta j)|^2\le\sum_{j\in\Z^n}|u(0,\delta j)|^2+\sum_{j\in\Z^n}|u(1,\delta j)|^2.\] 
\end{theorem}

\begin{proof}
 Let 
 \[\varphi_{0}\left(t\right)=\sum_{j\in\mathbb{Z}^n} u\left(0,\delta j\right) e^{2\pi ij\cdot t},\quad{\text{and}}\quad \varphi_{L}\left(t\right)=\sum_{j\in\mathbb{Z}^n} u\left(1,\delta j\right) e^{2\pi ij\cdot t}\] for $t\in[0,1]^n$. Then $\varphi_{0},\varphi_{L}\in L^{2}\left([0,1]^n\right)$.
 
For each $t\in[0,1]^n$ we define $q(t)$  such that $q(t)\ge 1$ and 
 \[q(t)+q(t)^{-1}=2(n+1)-2\sum_{l=1}^n\cos2\pi t_l.\] 
 More precisely
 $q(t)=\lambda(t)+\sqrt{\lambda^{2}(t)-1}$ and then $q(t)^{-1}=\lambda(t)-\sqrt{\lambda^{2}(t)-1},$ where 
 \[\lambda(t)=n+1-\sum_{l=1}^n\cos 2\pi t_l.\] 
 
Now for $k=1,...,L-1$ we consider  \[\varphi_{k}\left(t\right)=\frac{q(t)^k-q(t)^{-k}}{q(t)^L-q(t)^{-L}}\varphi_{L}\left(t\right)+\frac{q(t)^{L-k}-q(t)^{-L+k}}{q(t)^L-q(t)^{-L}}\varphi_{0}\left(t\right).\]
Since $q\ge 1$,  we have
\[q(t)^k-q(t)^{-k}\le  q(t)^L-q(t)^{-L},\quad{\text{and}}\quad 
q(t)^{L-k}-q(t)^{-L+k}\le q(t)^{L}-q(t)^{-L}.\]
 Then $\varphi_{k} \in L^{2}\left([0,1]^n\right)$ and $\|\varphi_k\|_2\le\|\varphi_0\|_2+\|\varphi_L\|_2$.
 Thus 
 \[\varphi_{k}\left(t\right)=\sum_{j\in\mathbb{Z}^n} v\left(k,j\right) e^{2\pi ij\cdot t},\] 
 where $\left\{v\left(k,j\right)\right\}_{j\in\mathbb{Z}^n}\in l^{2}\left( \mathbb{Z}^n\right).$
 Remark that \[q\left(t\right)=\frac{1+q^{2}\left(t\right)}{2\lambda\left(t\right)}\quad  {\text{and therefore}}\quad q^{k}\left(t\right)=\frac{q^{k-1}\left(t\right)+q^{k+1}\left(t\right)}{2\lambda\left(t\right)}.\] 
  Then \[\varphi_{k}\left(t\right)=\frac{\varphi_{k-1}\left(t\right)+\varphi_{k+1}\left(t\right)}{2\lambda\left(t\right)}\] and \[\varphi_{k}=\frac{1}{2(n+1)}\left[\varphi_{k+1}+\varphi_{k-1}+\varphi_{k}\left(\sum_{l=1}^ne^{2\pi it_l}+e^{-2\pi it_l}\right)\right].\] 
Hence the Fourier coefficients $v\left(k,j\right)$ satisfy 
  \[v\left(k,j\right)=\frac{1}{2(n+1)}\left[v\left(k+1,j\right)+v\left(k-1,j\right)+\sum_{l=1}^n\left(v\left(k,j-e_l\right)+v\left(k,j+e_l\right)\right) \right].\]
It means that $v$ is a discrete harmonic function on $[1,L-1]\times\Z^n$. We have that $v\left(0,j\right)=u\left(0,\delta j\right)$ and $v\left(L,j\right)=u\left(1,\delta j\right).$
  Note also that 
  \[\left|v\left(k,J\right)\right|^{2}\leq \sum_{j\in\mathbb{Z}^n}\left|v\left(k,j\right)\right|^{2}=\left\|\varphi_{k}\right\|^2_{L^{2}\left([0,1]^n\right)}\leq\left(\left\|\varphi_{0}\right\|_{L^{2}\left([0,1]^n\right)}+\left\|\varphi_{L}\right\|_{L^{2}\left([0,1]^n\right)}\right)^2.\] 
  Thus $v\left(k,J\right)$ is bounded, in particular $\left|v\left(k,j\right)\right|=o\left(\exp({b_{\delta}\|y\|_1})\right)$ when $\| y\|_1\rightarrow\infty.$ By Corollary in  2.4, \ref{c:m} we have $v\left(k,j\right)=u\left(\delta k,\delta j\right)$ and $\left\{u\left(\delta k,\delta j\right)\right\}_{j\in\mathbb{Z}^n}\in l^{2}\left(\mathbb{Z}^n\right)$ with the required estimate.
\end{proof}

\begin{remark}
We have also proved that if $u$ is a $\delta$-discrete harmonic function on $S^\delta$ that is square-summable along the hyperplanes $\{\delta k\}\times(\delta\Z)^n$ then there exist two functions $a_1,a_2\in L^2\left([0,1]^n\right)$ such that 
\begin{equation}\label{eq:st2}
u(\delta k,\delta j)=\int_{[0,1]^n}\left(a_1(t)q(t)^k+a_2(t)q(t)^{-k}\right)e^{-2\pi j\cdot t}dt,
\end{equation}
where  $q(t)\ge1$ and is defined by
\[
q(t)+q^{-1}(t)= 2(n+1)-2\sum_{l=1}^n\cos2\pi t_l.\] 
Reviewing the computations in the proof of the lemma, we see that
\[
a_1(t)=\frac{\varphi_{L}(t)-q(t)^{-L}\varphi_{0}(t)}{q(t)^{L}-q(t)^{-L}},\quad a_2(t)=\frac{q(t)^{L}\varphi_{0}(t)-\varphi_{L}(t)}{q(t)^{L}-q(t)^{-L}}. \]

Thus the theorem provides a constructive procedure for solution of the Dirichlet problem for tempered harmonic function in a strip as well as a stability estimate for this procedure.
\end{remark}

\noindent{\bf 3.2 Three line theorem for discrete harmonic functions}

In this subsection we prove a three line theorem for the gradients of discrete harmonic functions, the corresponding continuous result and its connections to the interpolation theory can be found in S. Janson and J. Peetre \cite{JP84}.
\begin{definition}
Let $u(x,y)$ be a $\delta$-discrete function on a subdomain of the lattice $(\delta\Z)^{n+1}$, its discrete partial derivatives are defined by
\[u_{x}(x,y)=\delta^{-1}\left(u\left(x+\delta,y\right)-u\left(x,y\right)\right)\quad {\text{and}}\] \[u_{y_l}(x,y)=\delta^{-1}\left(u\left(x,y+\delta e_l\right)-u\left(x,y\right)\right).\] 
\end{definition}
 For the case of the strip $S=[0,1]\times\R^n$ all discrete partial derivatives in $y$-variables are defined on the same domain, while $u_x$ is defined on $[0,1-\delta]\times\R^n$.
 \begin{definition}
 The discrete gradient of a discrete function $u(x,y)$ on a subdomain of the lattice $(\delta\Z)^{n+1} $ is defined as
 \[
 \nabla u(x,y)=\left(u_{x}(x,y),\  u_{y_{1}}(x,y),\ u_{y_{2}}(x,y),\ldots ,u_{y_{n}}(x,y)\right)\]
 \end{definition}
\begin{theorem}\label{th:3l}
Let $u$ be a $\delta$-discrete harmonic function in $\left[0,1\right]\times\R^n$, $\delta^{-1}=M+1$ for some positive integer $M$. Suppose that $u$ satisfies (\ref{eq:st1a}) and
\[\left\{u\left(0,\delta j\right)\right\}_{j\in\mathbb{Z}^n}\in l^{2}\left(\mathbb{Z}^n\right),\quad \left\{u\left(1,\delta j\right)\right\}_{j\in\mathbb{Z}^n}\in l^{2}\left(\mathbb{Z}^n\right).\] 
Let further 
\[m\left(k\right)=\delta^2\left\|u_{x}\left(\delta k,\delta j\right)\right\|^{2}_{l^{2}\left(\mathbb{Z}^n\right)}+\delta^2\sum_{l=1}^n\left\|u_{y_l}\left(\delta k,\delta j\right)\right\|^{2}_{l^{2}\left(\mathbb{Z}^n\right)}\ {\text{for}}\ k=0,\ 1,\ \ldots,\ M.\] 
Then
\[m\left(k\right)\leq \left(m\left(0\right)\right)^{1-\frac{k}{M}}\left(m\left(M\right)\right)^{\frac{k}{M}}.\]
\end{theorem}
\begin{proof}
Using  (\ref{eq:st2}) and the definition of the discrete partial derivatives, we get
\[
u_x(\delta k,\delta j)=\delta^{-1}\int_{[0,1]^n}\left(a_1(t)q(t)^k(q(t)-1)+a_2(t)q(t)^{-k}(q(t)^{-1}-1)\right)e^{-2\pi j\cdot t}dt,
\]
and 
\[
\|u_x(\delta k, \delta j)\|^2_{l^2(\Z^n)}=\delta^{-2}\|a_1(t)q(t)^{k}(q(t)-1)+a_2(t)q(t)^{-k}(q(t)^{-1}-1)\|^{2}_{L^2([0,1]^n)}.\]
Further,
\[
u_{y_l}(\delta k,\delta j)=\delta^{-1}\int_{[0,1]^n}\left(a_1(t)q(t)^k+a_2(t)q(t)^{-k}\right)e^{-2\pi j\cdot t}(e^{-2\pi i t_l}-1)dt,
\]
\[
\|u_{y_l}(\delta k, \delta j)\|^2_{l^2(\Z^n)}=\delta^{-2}\|(a_1(t)q(t)^k+a_2(t)q(t)^{-k})(e^{-2\pi it_l}-1)\|^{2}_{L^2([0,1]^n)}.\]
 Then, adding up the identities above, we get 
\begin{multline}\label{eq:m1} 
m(k)=\delta^2\left\|u_{x}\left(\delta k,\delta j\right)\right\|^{2}_{l^{2}\left(\mathbb{Z}^n\right)}+\delta^2\sum_{l=1}^n\left\|u_{y_l}\left(\delta k,\delta j\right)\right\|^{2}_{l^{2}\left(\mathbb{Z}^n\right)}=\\
\left\|a_{1}\left(t\right)q(t)^{k}\left(q(t)-1\right)+a_{2}\left(t\right)q(t)^{-k}\left(q(t)^{-1}-1\right)\right\|^{2}_{L^{2}\left([0,1]^n\right)}+\\
\sum_{l=1}^n\left\|a_{1}\left(t\right)q(t)^{k}\left(e^{-2\pi it_l}-1\right)+a_{2}\left(t\right)q(t)^{-k}\left(e^{-2\pi it_l}-1\right)\right\|^{2}_{L^{2}\left([0,1]^n\right)}.
\end{multline}
We note that $q(t)$ is real and by the definition $q(t)+q(t)^{-1}=2(n+1)-2\sum_{l=1}^n\cos2\pi t_l$, therefore
\begin{equation}\label{eq:q}
(q(t)-1)(q(t)^{-1}-1)=2\sum_{l=1}^n\cos2\pi t_l-2n=-\sum_{l=1}^n(e^{-2\pi i t_l}-1)(e^{2\pi i t_l}-1).
\end{equation}
Finally,
\begin{multline}\label{eq:m2}
\delta^2m(k)=\left\|a_{1}\left(t\right)q(t)^{k}\left(q(t)-1\right)\right\|^2_{L^2\left([0,1]^n\right)}+\left\|a_{2}\left(t\right)q(t)^{-k}\left(q(t)^{-1}-1\right)\right\|^{2}_{L^{2}\left([0,1]^n\right)}\\+
\sum_{l=1}^n\left\|a_{1}\left(t\right)q(t)^{k}\left(e^{-2\pi it_l}-1\right)\right\|^2_{L^2\left([0,1]^n\right)}+\left\|a_{2}\left(t\right)q(t)^{-k}\left(e^{-2\pi it_l}-1\right)\right\|^{2}_{L^{2}\left([0,1]^n\right)}.
\end{multline}
Each term in the right hand side of the last formula can be written in the form
$s(k)=\|b(t)q(t)^{\pm k}\|_2^2$ for some $b\in L^2([0,1]^n)$ and $q(t)^{\pm k}\in L^\infty([0,1]^n)$. 
By H\"{o}lder's inequality, we have 
\begin{multline*}
s(k)=\left\|b\left(t\right)q^{k}\left(t\right)\right\|^{2}_{L^{2}\left([0,1]^n\right)}
\leq\left(\int_{[0,1]^n}\left|b\left(t\right)\right|^{2}dt\right)^{1-\frac{k}{M}}\left(\int_{[0,1]^n}\left|b\left(t\right)\right|^2q^{2}(t)dt\right)^{\frac{k}{M}}\\
\leq \left(s(0)\right)^{1-\frac{k}{M}}\left(s(M)\right)^{\frac{k}{M}}.
\end{multline*}
 Applying the same computation for  each term and using the lemma below we conclude the proof of the theorem.
\end{proof}

\begin{lemma}
If each function $ m_{l}:\left[0,  1, \ldots,M\right]\rightarrow \mathbb{R_{+}}$ satisfies the inequality   
\[m\left(k\right)\leq\left[m\left(0\right)\right]^{1-\frac{k}{M}}\left[m\left(M\right)\right]^{\frac{k}{M}}\] 
then the sum $m(k)=\sum_l m_l(k)$ satisfies the same inequality. 
\end{lemma}
\begin{proof}
It suffices to prove the statement when $m(k)=m_1(k)+m_2(k)$ is the sum of two functions. Let $\alpha=k/M$ then we have 
\begin{multline*}
m(k)=m_1(k)+m_2(k)\le m_1(0)^{1-\alpha}m_1(M)^\alpha+m_2(0)^{1-\alpha}m_2(M)^{\alpha}=\\
m(0)^{1-\alpha}m(M)^\alpha
\left[\left(\frac{m_1(0)}{m(0)}\right)^{1-\alpha}\left(\frac{m_1(M)}{m(M)}\right)^\alpha+\left(\frac{m_2(0)}{m(0)}\right)^{1-\alpha}\left(\frac{m_2(M)}{m(M)}\right)^\alpha\right].\end{multline*}
And the lemma follows from the elementary inequality \[x^{1-\alpha}y^\alpha+(1-x)^{1-\alpha}(1-y)^\alpha\le 1\] when $x,y\in[0,1]$ and $\alpha\in[0,1]$.
\end{proof}

\begin{remark} The proof of Theorem \ref{th:3l} above is similar to that of the continuous three-line theorem, see  \cite{JP84}. In the continuous case the passage from (\ref{eq:m1}) to (\ref{eq:m2}) is trivial, in discrete case we fortunately have the identity (\ref{eq:q}).

For continuous harmonic functions similar three balls or three spheres theorems can be obtain, see for example J. Korevaar and  J. L. H. Meyers \cite{KM94}  and E.~Malinnikova \cite{Ma00}. There are no trivial generalizations of those results as a harmonic function can vanish on any finite square without being identically zero.
\end{remark}
\vspace{1.5cc}
\begin{center}
{\bf 4. HARMONIC MEASURE AND STABILITY ESTIMATES}
\end{center}

 In this section we  study $\h$-discrete harmonic functions that are defined on the cylinder $D^{\delta}(\Omega)=\Omega^{\delta}\times (\h\Z)$. Discrete harmonic measure on truncated cylinder is estimated first, then  we apply these estimates to give a more precise version of the Phragm\'{e}n-Lindel\"{o}f theorem and prove some stability results. 
 
 \noindent{\bf 4.1 Discrete harmonic measure} 

 Let now $\mathcal{H}_0(D^{\delta})$ denote the space of $\h$-discrete harmonic functions on $D^{\delta}(\Omega)$ that vanish on the boundary. Such function is uniquely determined by its values on two layers $\Omega^\h\times \{a\}$ and $\Omega^\h\times\{b\}$  (where it may attain arbitrary values) and the dimension of $\mathcal{H}_0(D^{\delta})$ equals $2K^\h$, where $K^\h$ is the number of points in $\Omega^\h$.

We note that for a function $u(x)=u(x',x_{n+1})$ on $D^{\delta}(\Omega)$ we have
\begin{multline*}
\Delta_{\h,n+1}u(x^{\prime},x_{n+1})=\\
\Delta_{\h,n}u(x^{\prime},x_{n+1})+\h^{-2}(u(x^{\prime},x_{n+1}+\h)+u(x^{\prime},x_{n+1}-\h)-2u(x^{\prime},x_{n+1})).
\end{multline*}
Let $\{f^\h_k\}_{k=1}^{K^\h}$ be a sequence of eigenfunctions of the Dirichlet problem for the Laplacian in $\Omega^\h$, discussed in 2.2. Then it is easy to check that the following functions form a basis for $\mathcal{H}_0(D^{\delta})$ 
\[u^\h_{k}(x)=f^\h_{k}(x^{\prime}) \cosh(a^\h_{k}x_{n+1}),\ \ v^\h_{k}(x)=f^\h_{k}(x^{\prime})\sinh(a^\h_{k}x_{n+1}),\ \ k= 1, 2, ...,K^\h,\]
where $a^\h_{k}$ is the only positive solution of 
\[\cosh \h a^\h_{k}=1+\frac{1}{2}\h^{2}\lambda^\h_{k}.\]
Now we calculate the discrete harmonic measure of the bases of a truncated cylinder. Let $g_N^{\h}$ be the $\h$-discrete harmonic function on  $D^{\delta}_{N}(\Omega)=\overline{\Omega}^{\delta}\times([-N,N]\cap\left(\h\Z\right))$ defined by its boundary values
\begin{equation*}
\left\{ \begin{array}{lll}
g_N^{\h}(x^{\prime},\pm N)&=1 \quad & x^{\prime}\in \Omega^{\h}\\
 g_N^{\h}(x^{\prime},x_{n+1})&=0\quad &x^{\prime}\in \partial \Omega^{\h},-N\le x_{n+1}\le N.
 \end{array}\right. 
\end{equation*}

\begin{lemma}\label{l:1}
The harmonic measure $g_N^\h(x)=g_N^\h(x',x_{n+1})$ is given by
\begin{equation*}
g_N^{\h}(x^{\prime},x_{n+1})=\sum^{K^{\h}}_{k=1}d^{\h}_{k}f^{\h}_{k}(x^{\prime})\frac{\cosh(a^{\h}_{k}x_{n+1})}{\cosh a^{\h}_{k}N},
\end{equation*}
where
$d_k^\h=\sum_{x'\in\Omega^\h}f_k^\h(x')$.
\end{lemma}
\begin{proof}
Clearly $g_N^\h$ is an even function with respect to $x_{n+1}$ and therefore it can be written as
\begin{equation}\label{eq:sw2}
g_N^{\h}(x^{\prime},x_{n+1})=\sum^{K^{\h}}_{k=1}C_{k}f^{\h}_{k}(x^{\prime})\cosh(a^{\h}_{k}x_{n+1}),
\end{equation}
where the coefficients $C_{k}$ satisfy the linear system of equations
\begin{equation*}
1=\sum^{K^{\h}}_{k=1}C_{k}f^{\h}_{k}(x^{\prime})\cosh(a^{\h}_{k}N),
\end{equation*}
for each $x^{\prime}\in \Omega^{\h}$. 
Since functions $\{f_k^\h\}_{k=1}^{K^\h}$ form an orthonormal basis,
we obtain
\begin{equation}\label{eq:sw6}
C_{k}\cosh a_k^\h N=\sum_{x^{\prime}}f^{\h}_{k}(x^{\prime})=d^{\h}_{k}.
\end{equation}
Substituting (\ref{eq:sw6}) in (\ref{eq:sw2}) we get the required formula.
\end{proof}

We conclude this subsection by one auxiliary inequality. We note that the values of the function $g^{\h}_{N}(x^{\prime},x_{n+1})$ on the middle hyperplane $\{x_{n+1}=0\}$ are given by
\[
g^{\h}_{N}(x^{\prime},0)=\sum^{K^{\h}}_{k=1}d^{\h}_{k}f^{\h}_{k}(x^{\prime})\frac{1}{\cosh a^{\h}_{k}N}.
\]
Then a linear combination of the values of $u$ on $\Omega^\h\times\{0\}$  admits the following estimate
\begin{multline}\label{eq:sumg}
\sum_{x^{\prime}}w(x')g^{\h}_{N}(x^{\prime},0)=\sum_{x^{\prime}}\sum^{K^{\h}}_{k=1}d^{\h}_{k}w(x')f^{\h}_{k}(x^{\prime})\frac{1}{\cosh a^{\h}_{k}N}\le\\
\sum^{K^{\h}}_{k=1}\frac{|d^{\h}_{k}|}{\cosh a^{\delta}_{k}N}\left(\sum_{x'}|w(x')|^2\right)^{1/2},
\end{multline}
we applied the Cauchy-Schwarz inequality and used that eigenfunctions $f_k^\h$ are normalized by $\sum_{x'}|f_k^\h(x')|^2=1$. 

\noindent{\bf 4.2 Phragm\'{e}n-Lindel\"{o}f theorem, improved version}

Now we  prove a version of the Phragm\'{e}n-Lindel\"{o}f  theorem for $\h$-discrete subharmonic functions in truncated cylinder $D^{\delta}_{N}(\Omega).$ We want to show that if a subharmonic function is positive inside the cylinder, say at some points on the section $\Omega^\h\times\{0\}$, then it grows at least exponentially.  Moreover, we can give estimates on the truncated cylinders and not only asymptotic result as in Theorem \ref{th:th1}. We use the following notation $u^+=\max\{0,u\}$.

\begin{theorem}\label{th:PLP}
Suppose $u$ is a $\h$-discrete subharmonic function on $D^{\delta}_{N}(\Omega)$ such that $u(x',x_{n+1})=0$ when $x'\in\partial\Omega^\h$ and $u$ satisfies the following positivity condition on $\Omega\times\{0\}$
\[
\sum_{x^{\prime} \in\Omega^{\h}}u^+(x',0)^2=A^2 K^\h>0. 
\]
Then
\begin{equation}\label{eq:PL1}
\max_{\Omega^{\h}\times\left[-N,N\right]}u(x^{\prime},x_{n+1})\geq \frac{A}2\left(\sum_{k}\exp(-a^{\h}_{k}N)\right)^{-1},
\end{equation}
where  $a^{\h}_{k}= \h^{-1}\cosh^{-1}(1+\frac{1}{2}\h^2\lambda^{\h}_{k}).$ In particular, there exists a constant $C_{\Omega}$ that depends only on $\Omega$ such that
\begin{equation}\label{eq:PL2}
\max_{\Omega^{\h}\times\left[-N,N\right]}u(x^{\prime},x_{n+1})\geq C_{\Omega}A\exp(a^{\h}_{1}N),\end{equation}
for any $N\in\N$ and any $\h<\h_0$.
\end{theorem}

The inequality (\ref{eq:PL1}) is more precise than (\ref{eq:PL2}). We write the constant explicitly and, as soon as $\lambda_k^\h$ are known, the right hand side of (\ref{eq:PL1}) can be estimated. Clearly, the right hand side of (\ref{eq:PL1}) is of order $\exp(a_1^\h N)$ when $N\rightarrow\infty$. This is expressed accurately in inequality (\ref{eq:PL2}). The constant $C_\Omega$ is not explicit, but it depends neither on $N$ nor  on $\h$, so we can also fix $N$ and let $\h$ go to zero to get estimates of continuous functions that can be approximated by discrete subharmonic ones.

\begin{proof}
 Let  $M_N=\max_{\left|x_{n+1}\right|=N}u(x^{\prime},x_{n+1})$. 
Then by the maximum principle, 
\[u(x^{\prime},x_{n+1})\leq M_{N}  g^{\h}_{N}(x^{\prime},x_{n+1})\quad{\text{on}}\quad \Omega^{\h}\times\left[-N,N\right],\] where $g^{\h}_{N}$ is the harmonic measure from Lemma \ref{l:1}, clearly $g^\h_{N}\geq 0$.
 Taking the  linear combination over $x^{\prime}\in\Omega^\h$ with non-negative coefficients $w(x')=u^+(x',0)$ and using (\ref{eq:sumg}), we obtain
\[\sum_{x'}u^+(x',0)^2=\sum_{x^{\prime}}u^+(x',0)u^+(x^{\prime},0)\leq M_{N}\sum^{K^{\h}}_{k=1}\frac{|d^{\h}_{k}|}{\cosh a^{\h}_{k}N}\left(\sum_{x'}|u^+(x',0)|^2\right)^{1/2}.\]
Then we have
\begin{equation*}
M_N\geq (\sum_{x'}u^+(x',0)^2)^{1/2}\left(\sum^{K^{\h}}_{k=1}\frac{|d^{\h}_{k}|}{\cosh a^{\h}_{k}N}\right)^{-1}=A(K^{\h})^{1/2}\left(\sum^{K^{\h}}_{k=1}\frac{|d^{\h}_{k}|}{\cosh a^{\h}_{k}N}\right)^{-1}.
\end{equation*}
Applying the Cauchy-Schwarz inequality, we get
\[|d^{\h}_{k}|=\left|\sum_{x^{\prime}}f^{\h}_{k}(x^{\prime})\right|\leq \left(\sum_{x^{\prime}}(f^{\h}_{k}(x^{\prime}))^{2}\right)^{\frac{1}{2}}\left(\sum_{x^{\prime}}1\right)^{\frac{1}{2}}\leq (K^{\h})^{\frac{1}{2}}. \]
Now, we combine the last two inequalities and obtain
\[M_{N}\geq A\left(\sum^{K^{\h}}_{k=1}\frac{1}{\cosh a^{\h}_{k}N}\right)^{-1}.\] 
Then (\ref{eq:PL1}) follows from the following inequality
\[
\sum_{k=1}^{K^\h}\frac{1}{\cosh a^{\h}_{k}N}\leq 2\sum_{k=1}^{K^\h}\exp(-a_k^{\h}N).
\] 

To prove (\ref{eq:PL2}) we may assume that $\h$ is small (otherwise we have an upper bound for $K^\h$). We  partition the eigenvalues $\lambda_{k}^\h$ into two groups. We choose a positive number $c$ and define 
$I_1=\{k:\lambda_k^\h<c \h^{-2}\}$ and $I_2=\{k:\lambda_k^\h\ge c\h^{-2}\}$. Let also $c_0=\cosh^{-1}(1+c)$, then
\[
\sum_{k\in I_2}\exp(-a_k^\h N)\le \sum_{k\in I_2}\exp(-\h^{-1}c_0N)\le
K^\h\exp(-\h^{-1}c_0N)\le C_0\exp(-a_1^\h N),
\]
when $\h$ is small enough, since $K^\h\le C\h^{-n}$ and $a_1^\h\rightarrow\left(\lambda_1(\Omega)\right)^{1/2}$ as $\h\rightarrow 0$.

For the second part of the sum we have $\h\sqrt{\lambda_k^\h}<c$. We consider the function $\alpha:\R_+\rightarrow\R_+$ defined by
\[
\cosh\alpha(s)=1+\frac{1}{2}s^2.\]
Then $a_k^\h=\h^{-1}\alpha(\h\sqrt{\lambda_k^\h})$ and a simple calculation gives
\[
\alpha'(s)=\frac{2}{\sqrt{4+s^2}}.\]  
Denoting the minimum of the derivative of $\alpha$  on $[0,c]$ by $d$, we obtain
\[
a_k^\h\ge a_1^\h+d\left((\lambda_k^\h)^{1/2}-(\lambda_1^\h)^{1/2}\right).
\]
 
Now we partition $I_1$ further into $J_l=\{k: l\le(\lambda_k^\h)^{1/2}-(\lambda_1^\h)^{1/2}< l+1\}$, $l=0,1,...$ and let $|J_l|$ denote the cardinality of $J_l$. We consider any cube $Q$ such that $\Omega\subset Q$  and apply inequalities 
(\ref{eq:incl1}) and (\ref{eq:Wa}) to obtain 
\[|J_l|\le N^{\delta}_{\Omega}\left(\left((\lambda^{\delta}_{1})^{\frac{1}{2}}+l+1\right)^{2}\right)
\leq N^{\delta}_{Q}\left(\left((\lambda^{\delta}_{1})^{\frac{1}{2}}+l+1\right)^{2}\right)\le C_\Omega(l+1)^n,\]
for each $l=0,1,... .$
Finally, we get
\begin{multline*}
\sum_{k\in I_1}\exp(-a_k^\h)\le\sum_{l=0}^\infty\sum_{k\in J_l}\exp(-a_k^\h N)\le
\sum_{l=0}^\infty\exp(-(a_1^\h+ld)N)|J_l|\le\\
C_{\Omega}\exp(-a_1^\h N)\sum_{l=0}^\infty(l+1)^n\exp(-ldN).
 \end{multline*}
The last sum is finite and can be bounded by a constant independent of $N\in\N$ and $\h$. This concludes the proof of the theorem.
 \end{proof}
One of the differences between the continuous and discrete cases lies in the formulas connecting eigenvalues $\lambda$ and corresponding numbers $a$. For the continuous case one has $a(\lambda)=\sqrt{\lambda}$ while for the discrete case the formula becomes
\[a^\delta(\lambda)=\delta^{-1}\cosh^{-1}(1+\frac{1}{2}\delta^2\lambda).\]
This function resembles $\sqrt{\lambda}$ on the interval $[0,c\delta^{-2}]$ but grows as $\log\lambda$ when $\lambda\rightarrow\infty$. To deal with the discrete case we have partitioned the set of eigenvalues into two parts. 

\noindent{\bf 4.3 Stability estimates for solution of the Dirichlet problem}

A standard argument shows that estimates of the harmonic measure imply conditional stability estimates for harmonic function. We apply it for truncated cylinders and  prove the following 
\begin{theorem}\label{th:last}
Let $h$ be a $\delta$-discrete harmonic function on $D^{\delta}_{N}(\Omega)$ with boundary values $f$ on $\partial\Omega^\h\times[-N,N]$ and such that $|h(x',\pm N)|\le M_N.$ Then
\begin{equation}\label{eq:last}
\max_{x'}|h(x',0)|\le \max|f|+C_\Omega (M_N+\max|f|)\exp(-a_1^\delta N).
\end{equation}
In particular, if  $h$ is harmonic in $D^{\delta}(\Omega)$, \[|h(x',x_{n+1})|=o(\exp(a_1^\delta|x_{n+1}|)) \quad{when} \ |x_{n+1}|\rightarrow\infty \] and $h$ is bounded on the boundary $\partial\Omega\times(\delta\Z)$ then $h$ is bounded by the same constant in $D^{\delta}(\Omega).$
\end{theorem}
\begin{proof}
Let $v_N$ be the $\delta$-discrete  harmonic function in the truncated cylinder $D^{\delta}_{N}(\Omega)=(\Omega\times(-N,N))^\h$  that solves the following Dirichlet problem
\[
\Delta_{n+1,\delta}v=0,\quad v(x',\pm N)=0,\ x'\in\Omega^\h,\quad{\text{and}}\quad v(x',x_{n+1})=f(x',x_{n+1}),\ x'\in\partial\Omega^\h.\]
By the maximum principle for the bounded domain $D^{\delta}_{N}(\Omega),$  $|v(x)|\le\max|f|$.  
Then $u=h-v$  is $\delta$-discrete harmonic function on $D^{\delta}_{N}(\Omega)$ that vanishes on the part $\partial\Omega^\h\times[-N,N]$ of the boundary and satisfies 
\[
\max_{\Omega^\delta\times[-N,N]}| u(x',x_{n+1})|\le \max|f|+M_N.\]
We compare it to a multiple of the harmonic measure $g_N^\delta$ and  use the estimate 
\[|g^{\delta}_N(x',0)|\le C_\Omega\exp(-a_1^\delta N)\]
that follows from the proof of  Theorem \ref{th:PLP}. Then we obtain
\[
|u(x',0)|\le C_\Omega (M_N+\max|f|)\exp(-a_1^\delta N). \]
This implies (\ref{eq:last}). The second statement of the theorem follows from (\ref{eq:last}).
\end{proof}

{\bf Acknowledgments}
The work was supported by the
Research Council of Norway grant 185359/V30.

 The author is grateful to both referees for their valuable remarks and suggestions and would like to thank his PhD supervisor Eugenia Malinnikova for her constructive comments and encouragement during the preparation of the paper.

\vspace{1cc}

 Department of Mathematics,\newline
  Norwegian University of Science and Technology,\newline
  7491, Trondheim, Norway,\newline
 email: guadie@math.ntnu.no
 {\small
\noindent}
\end{document}